\documentclass[11pt]{amsart}
\usepackage{amsmath,amssymb,a4wide} 
\usepackage{xcolor}
\usepackage{mathrsfs}
\usepackage{dsfont}
\usepackage{mathabx}
\usepackage{tikz,pgfplots}
\usepackage{hyperref}
\usepackage{ulem}
\usepackage{graphicx}
\usepackage{svrsymbols}

\newtheorem{theorem}{Theorem}[section]
\newtheorem{lemma}{Lemma}[section]
\newtheorem{proposition}{Proposition}[section]
\newtheorem{remark}{Remark}[section]

\newtheorem{definition}{Definition}[section]
\newtheorem{example}{Example}[section]

\newcommand{\R}{{\mathbb R}} 
\newcommand{\C}{{\mathbb C}} 
\newcommand{\N}{{\mathbb N}}

\newcommand{\dd}{{\rm d}}

\newcommand{\icomplexe}{{\rm i}}

\usepackage{color}

\title[$\hat{A}$- and $\hat{I}$-stability of Runge--Kutta collocation methods]
{$\hat{A}$- and $\hat{I}$-stability of collocation Runge--Kutta methods}

\author[G. Dujardin]{Guillaume Dujardin}
\address[G. Dujardin]{Univ. Lille, Inria, CNRS, UMR 8524 - Laboratoire Paul Painlev\'e, F-59000 Lille, France\\
\href{mailto:guillaume.dujardin@inria.fr}{\tt guillaume.dujardin@inria.fr}}
\author[I. Lacroix-Violet]{Ingrid Lacroix-Violet}
\address[I. Lacroix-Violet]{Université de Lorraine, CNRS, IECL, F-54000 Nancy, France\break
\href{mailto:ingrid.lacroix@univ-lorraine.fr}{\tt ingrid.lacroix@univ-lorraine.fr}}

\begin{document}


\begin{abstract}
  This paper deals with stability of classical Runge--Kutta collocation methods.
  When such methods are embedded in linearly implicit methods as developed in \cite{DL2020}
  and used in \cite{DL2023PDE} for the time integration of nonlinear evolution PDEs,
  the stability of these methods has to be adapted to this context.
  For this reason, we develop in this paper several notions of stability, that we analyze.
  We provide sufficient conditions that can be checked algorithmically to ensure
  that these stability notions are fulfilled by a given Runge--Kutta collocation method.
  We also introduce examples and counterexamples used in \cite{DL2023PDE}
  to highlight the necessity of these stability conditions in this context.
\end{abstract}


\maketitle
{\small\noindent 




\section{Introduction}

This paper deals with stability of classical Runge--Kutta collocation methods.
We introduce and analyze several notions of stability for such methods,
that are, in particular, useful in the context of linearly implicit methods
for the numerical integration of evolution PDEs.
Indeed, the linearly implicit methods developed in \cite{DL2023PDE} embed
a Runge--Kutta collocation method.
Using additional unknowns and a collocation trick
(such as in the relaxation method developed in \cite{Besse2004} for the Schr\"odinger equation),
they allow for the solution of a {\it linear} system at each time step, even if the original
PDE is nonlinear.
Moreover, they are shown (in \cite{DL2023PDE}) to have high order accuracy in time,
provided they satisfy suitable stability conditions.
The goal of this paper is to introduce these stability notions,
to provide sufficient conditions on the coefficients of a Runge--Kutta collocation method
to ensure them, and to provide the reader with examples and counterexamples.

For solving dissipative problems such as the heat equation with such linearly implicit methods,
the relevant notion of stability for Runge--Kutta collocation methods that we introduce is named
$\hat A$-stability (see Definition \ref{def:stab} below).
This stability notion is a generalization of the classical $A$-stability
(which is recalled below in Section \ref{sec:StabRK}).
For solving oscillatory problems such as the Schr\"odinger equation with such methods,
the relevant notion of stability for Runge--Kutta collocation methods that we introduce is
named $\hat I$-stability.

Of course, stability analysis of Runge--Kutta methods has a long history.
Let us mention the $A$-stability introduced by Dahlquist \cite{Dahlquist63},
to ensure appropriate long-time behaviour (the letter $A$ is short for the word ``asymptotic'')
of Runge--Kutta methods applied to non-expensive linear scalar equations,
the $B$-stability introduced by Butcher \cite{Butcher75} \cite{BurrageButcher79}
for vector fields satisfying a one-sided Lipschitz condition,
which generalizes asymptotic stability to nonlinear differential equations
(see also \cite{Crouzeix79}).
The relations between these two notions is studied and summarized in \cite{Hairer86}
using the $W$-transform.
The notions of $A$- and $B$-stability have been generalized to non-autonomous ODEs
yielding the definitions of $AN$- and $BN$-stability \cite{BurrageButcher79}.
For Runge--Kutta methods, the notion of $BN$-stability is essentially the same as that
of algebraic stability \cite{Hairer81}.
The definition of algebraic stability involves a non-negative quadratic form depending
on the coefficients of the Runge--Kutta method at hand.
It was discovered in \cite{SanzSerna88} that if this quadratic form is zero, then the Runge--Kutta
method is symplectic when applied to Hamiltonian equations.
The interested reader may find a summary of stability notions for numerical methods
for solving evolution ODEs in \cite{Butcher06}, including $G$-stability.
Stability properties of Runge--Kutta methods have also been studied outside
the ODE context (see for example \cite{Song12}).

In the context of integrating time-evolution systems arising from the spatial discretization
of evolution PDEs, the relevant notions of stability depend on the PDE at hand,
and, in particular, on whether that original PDE is parabolic, hyperbolic, or hamiltonian.
This dates back, at least, to the seminal work of Crouzeix
(see for example \cite{Crouzeix75} for general ODEs involving operators and
\cite{Crouzeix76} for parabolic problems),
and has a continued history in both linear and nonlinear contexts
(see for example \cite{Crouzeix93} for general
rational approximations of analytic semigroups in Banach spaces,
\cite{Lubich95} for quasilinear parabolic problems,
\cite{Tadmor01} for strong stability preserving high order methods for hyperbolic problems).

In this paper, we investigate notions of stability for collocation Runge--Kutta methods
embedded in linearly implicit methods defined in the ODE setting in \cite{DL2020}
and later extended to the semilinear PDE setting in \cite{DL2023PDE}.
We focus on the PDE case, and assume the linearly implicit method is applied to the (large) ODE
system obtained after space discretization of the PDE at hand.
We distinguish two main cases, depending on the unbounded linear differential operator of the PDE.
First, when the linear operator is dissipative (roughly speaking, its spectrum
is located in a left half plane of the complex plane and the real part of that spectrum
is not bounded below), we introduce the notion of $\hat{A}$-stability.
Second, when the linear operator is dispersive (roughly speaking, its
spectrum is located in a vertical line of the complex plane and its imaginary part is not bounded),
we introduce the notion of $\hat{I}$-stability.
The relevance of these two stability notions for Runge--Kutta methods embedded in linearly
implicit method for solving evolution PDEs is discussed and numerically illustrated
in \cite{DL2023PDE}.
The goal of this paper is to provide sufficient conditions on the coefficients of the
collocation Runge--Kutta method to ensure $\hat{A}$- and $\hat{I}$-stability.

The outline of this paper is as follows.
In Section \ref{sec:StabRK}, we introduce several notions of stability for Runge--Kutta methods
that are relevant in our context, and recall the basics about Runge--Kutta collocation methods.
In Section \ref{sec:spA}, we indicated how one can compute the spectrum of the matrix of
the coefficients of a Runge--Kutta collocation method (Lemma \ref{lem:transfoTau}),
and we derive some sufficient conditions to ensure the stability properties defined in the previous
section (Lemmas \ref{lem:AdonneAchapeau} and \ref{lem:IdonneIchapeau}).
As an application of these criteria, we prove that Runge--Kutta collocation methods using
Gauss points are $\hat A$-stable (Theorem \ref{th:GaussAchapeau}),
and we provide an algorithm implementing our sufficient conditions
to guarantee that a given Runge--Kutta collocation method is $\hat A$-stable
(see Section \ref{subsec:algo}) which uses the Routh-Hurwitz algorithm.
We describe the influence of symmetry (with respect to $1/2$) of the collocation points
on the stability of a Runge--Kutta collocation method (Section \ref{subsec:symmetry}).
We also investigate the stability of (forward) Runge--Kutta collocation methods
with $s\leq 4$ stages in Section \ref{subsec:forward1234}.
The last section of this paper (Section \ref{sec:examples}) is devoted to the analysis
of a few examples of Runge--Kutta collocation method with at most $5$-stages that are used
numerically as examples or counterexamples in \cite{DL2023PDE}.



\section{Notions of stability of Runge--Kutta methods}
\label{sec:StabRK}

\subsection{Several notions of stability}
\label{subsec:stabRK}
Assume we are given a time interval $J\subset\R$, an open set $\Omega\subset\R^d$,
a (smooth) function $f$ from $J\times\Omega$ to $\R^d$ and $(t_0,y_0)\in J\times\Omega$.
For the time integration of the Cauchy problem
\begin{equation}
  \label{eq:CauchyEDO}
  \left\{
    \begin{matrix}
      y'(t) & = & f(t,y(t))\\
      y(t_0) & = & y_0
    \end{matrix},
  \right.
\end{equation}
one may use a Runge--Kutta method with $s\in\N\setminus\{0\}$ stages, defined by
$c_1,\cdots,c_s\in\R$, $(a_{i,j})_{1\leq i,j\leq s} \in \R^{s^2}$ and $b_1,\cdots,b_s\in\R$
defined by the following classical algorithm.
Given some (small) time step $h>0$, and starting from some $y^0\in\Omega$ close to $y_0$, compute
successively for $n\geq 0$ first $(y_{n,i})_{1\leq i\leq s}\in \Omega^s$ as the solution of
the nonlinear system
\begin{equation}
  \label{eq:RKint}
  y_{n,i} = y_n + h\sum_{j=1}^s a_{i,j} f(t_0+(n+c_j)h,y_{n,j}),
  \qquad
  (1\leq i\leq s),
\end{equation}
and then $y_{n+1}$ using the explicit computation
\begin{equation}
  \label{eq:RKexpl}
  y_{n+1} = y_n + h\sum_{i=1}^s b_i f(t_0+(n+c_i)h,y_{n,i}).
\end{equation}

These methods have a long history, which started at the end of the 19th century
\cite{Runge1895}, and have been extensively studied in several contexts (see for example
\cite{HNW93}, \cite{HW96} and references therein, as well as for several other numerical methods)
and order conditions as well as stability conditions are well understood.
Usually, the coefficients $(a_{i,j})_{1\leq i,j\leq s}$ are grouped in a square matrix denoted $A$,
$b_1,\cdots,b_s$ are grouped in a column vector denoted $b$ and $c_1,\cdots,c_s$
are grouped in a column vector denoted $c$.
In particular, the \emph{linear stability function} $R$ of these methods, defined for $\lambda\in\C$ by
\begin{equation}
\label{eq:defR}
  R(\lambda) = 1 + \lambda b^t (I-\lambda A)^{-1} \mathds{1},
\end{equation}
where $I$ is the identity matrix of size $s$ and $\mathds{1}$ is the column vector
of size $s$ with all coefficients equal to $1$, plays a crucial role in the analysis of most
forms ot stability for the methods. Indeed, in the scalar case $d=1$,
for the Dahlquist equation, which corresponds to \eqref{eq:CauchyEDO} with $\Omega=\R$
and $f(y,y)=ay$ for some $a\in\C$,
the sequence produced by the Runge--Kutta method \eqref{eq:RKint}-\eqref{eq:RKexpl} satisfies
for all $n\in\N$, $y^n=\left(R(ah)\right)^n y^0$.
Observe that $R$ is a rational function of $\lambda$ and any of its poles
in $\C$ is the inverse of a non-zero eigenvalue of the matrix $A$.

For the purpose of integrating systems of the form \eqref{eq:CauchyEDO} arising
from the space-discretization of semilinear evolution PDEs
using methods involving Runge--Kutta methods of the form \eqref{eq:RKint}-\eqref{eq:RKexpl}
(see \cite{DL2023PDE}),
and following \cite{Crouzeix80}, \cite{Hairer82} and \cite{BHV86},
we use the following eight notions of stability.
We denote by $\C^-$ the set of complex numbers with non positive real part,
and by $i\R$ the set of purely imaginary numbers.

\begin{definition}
\label{def:stab}
A Runge--Kutta method is said to be
\begin{itemize}
  \item \emph{$A$-stable} if for all $\lambda\in\C^-$, $|R(\lambda)|\leq 1$.
  \item \emph{$I$-stable} if for all $\lambda\in i\R$, $|R(\lambda)|\leq 1$.
  \item \emph{$AS$-stable} if for all $\lambda\in\C^-$, the matrix $(I-\lambda A)$ is
    invertible and $\lambda\mapsto \lambda b^t(I-\lambda A)^{-1}$ is uniformly bounded on $\C^-$.
  \item \emph{$ASI$-stable} if for all $\lambda\in\C^-$, the matrix $(I-\lambda A)$ is
    invertible and $\lambda\mapsto (I-\lambda A)^{-1}$ is uniformly bounded on $\C^-$.
  \item \emph{$IS$-stable} if for all $\lambda\in i\R$, the matrix $(I-\lambda A)$ is
    invertible and $\lambda\mapsto \lambda b^t(I-\lambda A)^{-1}$ is uniformly bounded on $i\R$.
  \item \emph{$ISI$-stable} if for all $\lambda\in i\R$, the matrix $(I-\lambda A)$ is
    invertible and $\lambda\mapsto (I-\lambda A)^{-1}$ is uniformly bounded on $i\R$.
  \item \emph{$\hat{A}$-stable} if it is $A$-stable, $AS$-stable and $ASI$-stable.
  \item \emph{$\hat{I}$-stable} if it is $I$-stable, $IS$-stable and $ISI$-stable.
\end{itemize}
\end{definition}

A few examples of Runge--Kutta methods and the analysis of their stability properties
can be found in Section \ref{sec:examples}.

\begin{remark}
  \label{rem:AdonneI}
  Observe that, from this definition, an $A$-stable Runge--Kutta method is also $I$-stable.
  Similarly, an $AS$-stable Runge--Kutta method is also $IS$-stable, and
  and $ASI$-stable Runge--Kutta method is also $ISI$-stable.
  As a consequence, an $\hat A$-stable method is also $\hat I$-stable.
\end{remark}

\begin{remark}
  \label{rem:SIimpliqueS}
  If $b$ is in the range of $A^t$, then $ASI$-stability implies $AS$-stability using Lemma
  4.4 of \cite{BHV86}. Similarly, $ISI$-stability implies $IS$-stability.
\end{remark}

\subsection{Collocation Runge--Kutta methods}
\label{subsec:collocationRK}

This paper deals with collocation Runge--Kutta methods in the sense of the following definition
and the analysis of their stability properties in the sense of Definition \ref{def:stab}.

\begin{definition}[collocation methods]
  \label{def:colmeth}
  The Runge--Kutta method defined by $A$, $b$ and $c$ is said to be of \emph{collocation}
  when $c_1,c_2,\dots,c_s$ are distinct real numbers
  and $A$ and $b$ satisfy for all $(i,j)\in\{1,\dots,s\}^2$,
  \begin{equation}
    \label{eq:defaijbi}
    a_{ij} = \int_0^{c_i} \ell_j(\tau) \dd \tau
    \qquad\text{and}\qquad
    b_i = \int_0^1 \ell_i(\tau)\dd \tau,
  \end{equation}
  where $(\ell_i)_{1\leq i\leq s}$ are the $s$ Lagrange polynomials associated to $(c_1,\dots,c_s)$
  defined by
  \begin{equation}
    \label{eq:defli}
    \forall i\in\{1,\dots,s\},\qquad
    \ell_i(\tau) = \prod_{\substack{j\neq i\\ j=1} }^s \frac{(\tau-c_j)}{(c_i-c_j)}. 
  \end{equation}
  We denote by $\pi$ the normalized polynomial of degree $s$, the zeros of which are $c_1,\cdots,c_s$,
  {\it i.e.}
  \begin{equation}
    \label{eq:defpi}
    \pi=\prod_{i=1}^s (X-c_i).
  \end{equation}
  Moreover, a Runge--Kutta collocation method is said to be \emph{forward} when
  all $(c_i)_{1\leq i\leq s}$ are non negative.
\end{definition}

Order condition for collocation Runge--Kutta methods are well known (see {\it e.g.} \cite{GNI2002}),
and, amongst those methods, Gauss methods are known to have {\it superconvergence} in the sense
that they have $s$ stages and order $2s$ in the ODE setting \cite{GNI2002}.
These collocation Runge--Kutta methods appear as an essential tool of the high oder linearly implicit
methods derived for ODEs in \cite{DL2020} and used further for PDEs in \cite{DL2023PDE}.
Since all the coefficients of such methods are defined once the distinct $s$ real numbers
$c_1,\cdots,c_s$ are chosen, their stability analysis (in the sense of Definition \ref{def:stab})
can be reduced to the analysis of polynomials with one unknown (whose coefficients are polynomials
in $c_1,\cdots,c_s$), as we shall see in Section \ref{sec:stabcollocationRK}.
One has for example the following lemma for the computation of $(a_{i,j})_{1\leq i,j\leq s}$.

\begin{lemma}[Expression of $A$ for collocation Runge--Kutta methods]
  \label{lem:relationA}
  Let $s\geq 1$ be an integer, and $c_1,\cdots,c_s$ be $s$ distinct real numbers.
  The corresponding Runge--Kutta collocation method with matrix $A$ defined in
  \eqref{eq:defaijbi}-\eqref{eq:defli} satisfies
  \begin{equation}
    \label{eq:relationA}
    A=W_c \times V_{c}^{-1},
  \end{equation}
  where
  \begin{equation}
    \label{eq:defVc}
    V_c=
    \begin{pmatrix}
      1 & c_1^1 & \cdots & c_1^{s-1} \\
      1 & c_2^1 & \cdots & c_2^{s-1} \\
      \vdots & \vdots & \ddots & \vdots \\
      1 & c_s^1 & \cdots & c_s^{s-1}
    \end{pmatrix}
    \qquad\text{and}\qquad
    W_c=
    \begin{pmatrix}
      \frac{c_1^1}{1} & \frac{c_1^2}{2} & \cdots & \frac{c_1^{s}}{s} \\
      \frac{c_2^1}{1} & \frac{c_2^2}{2} & \cdots & \frac{c_2^{s}}{s} \\
      \vdots & \vdots & \ddots & \vdots \\
      \frac{c_s^1}{1} & \frac{c_s^2}{2} & \cdots & \frac{c_s^{s}}{s}
    \end{pmatrix}.
  \end{equation}
\end{lemma}

\begin{proof}
  Since the $s$ quadrature points $c_1,\cdots,c_s$ are distinct, the matrix $A$ is well-defined
  by \eqref{eq:defaijbi}-\eqref{eq:defli}. Morover, we have for all polynomial fonction $f$ or
  degree at most $s-1$ that
  \begin{equation*}
    \forall i\in\{1,\cdots,s\},\qquad
    \sum_{j=1}^s a_{i,j} f(c_j) = \int_0^{c_i} f(\tau){\rm d} \tau.
  \end{equation*}
  In particular, for all $p\in\{1,\cdots,s\}$ and all $i\in\{1,\cdots,s\}$,
  for the function $f:\tau\mapsto \tau^{p-1}$, we obtain
  \begin{equation*}
    \sum_{j=1}^s a_{i,j} c_j^{p-1} = \frac{c_i^p}{p}.
  \end{equation*}
  This relation exactly states that $A\times V_c=W_c$. The inversibility of the Vandermonde
  matrix $V_c$ ensures that \eqref{eq:relationA} holds.
\end{proof}

\section{Stability analysis of collocation Runge--Kutta methods}
\label{sec:stabcollocationRK}

In this section, we consider a collocation Runge--Kutta method, the coefficients
$(a_{i,j})_{1\leq i,j\leq s}$ and $(b_i)_{1\leq i \leq s}$ of which are derived using Definition
\ref{def:colmeth} from the $s$ distinct real numbers $c_1,\cdots,c_s$.
Without loss of generality (using a permutation of the coefficients), we assume that
$c_1<\cdots<c_s$ are ordered this way.

\subsection{Computation of the characteristic polynomial of $A$ in terms of $(c_i)_{1\leq i\leq s}$}
\label{sec:spA}

\begin{lemma}
  \label{lem:transfoTau}
  Let $A$ be the matrix of a Runge--Kutta collocation method at the $s$ distinct real
  points $(c_i)_{1\leq i\leq s}$ with $c_1<\cdots<c_s$.
  Let $\pi$ denote the polynomial defined in \eqref{eq:defpi} in Definition \ref{def:colmeth}.
  Denote by $\tau$ the linear mapping
  \begin{equation*}
    \tau :
    \begin{pmatrix}
      \C_s[X] & \longrightarrow & \C_s[X]\\
      \sum_{k=0}^s a_k X^k & \longmapsto & \sum_{k=0}^s k! a_k X^k
    \end{pmatrix}.
  \end{equation*}
  The characteristic polynomial $\chi_A$ of $A$ is $\tau(\pi)/(s!)$.
\end{lemma}

\begin{proof}
  Assume that all $c_i\neq 0$ for the time being. For a polynomial
  \begin{equation}
    \label{eq:defP}
    P(X) = s \alpha_{s-1}X^{s-1} + \dots + 2\alpha_1X + \alpha_0 \in \C_{s-1}[X],
  \end{equation}
  we denote by ${\neutrino}(P)$ the column vector with component $i$ equal to $P(c_i)$.
  Introducing the diagonal matrix $D$ with entries $(1,2,\cdots,s-1,s)$, this yields
  \begin{equation}
    \label{eq:defv}
    {\neutrino}(P) = V_c D
    \left[
      \begin{matrix}
        \alpha_0\\
        \vdots\\
        \alpha_{s-1}
      \end{matrix}
    \right],
  \end{equation}
  where $V_c$ is the square matrix defined in \eqref{eq:defVc}.
  Note that $V_c$ is invertible since all $c_i$ are distinct.
  We extend the definition of $\neutrino$ to $\C_{s}[X]$. Note however that \eqref{eq:defv} no
  longer holds for polynomials of degree $s$.
  Conversely, for $Y=(y_1,\cdots,y_s)^t\in\C^s$, we denote by ${\mathfrak c}(P)$ the coefficients
  in the canonical basis of $\C_{s-1}[X]$ of the only polynomial of degree at most $s-1$,
  the value of which at $c_i$ is $y_i$ for all $i\in\{1,\cdots,s\}$.
  This corresponds to expanding the polynomial $\sum_{i=1}^s y_i \ell_i(X)$ in the canonical
  basis of $\C_{s-1}[X]$, where ${\ell_i}_{1\leq i\leq s}$ are the $s$ Lagrange interpolation
  polynomials defined in \eqref{eq:defli} at points $(c_1,\cdots,c_s)$.
  We associate to a polynomial $P\in\C_{s-1}[X]$ of the form \eqref{eq:defP} its primitive $Q$
  which vanishes at $0$. We have
  \begin{equation}
    \label{eq:defQ}
    Q(X) = \alpha_{s-1} X^s + \dots + \alpha_1 X^2 + \alpha_0X.
  \end{equation}
  For a polynomial $Q$ of this form, we define $\mathfrak c^1 (Q)$ as the column vector
  with entry $i$ equal to $\alpha_{i-1}$ for all $i\in\{1,\cdots,s\}$.
  Introducing the matrix
  \begin{equation}
    \label{eq:defV}
    V=
    \begin{pmatrix}
      c_1 & c_1^2 & \cdots & c_1^{s} \\
      c_2 & c_2^2 & \cdots & c_2^{s} \\
      \vdots & \vdots & \ddots & \vdots \\
      c_s & c_s^2 & \cdots & c_s^{s}
    \end{pmatrix},
  \end{equation}
  we have
  \begin{equation*}
    \neutrino (Q) = V \mathfrak c^1(Q).
  \end{equation*}
  Moeover, since $P$ is a derivative of $Q$, we also have
  \begin{equation*}
    \mathfrak c (P) = D \mathfrak c^1(Q).
  \end{equation*}
  For a general polynomial $P$ of the form \eqref{eq:defP}, the associated polynomial $Q$ as
  in \eqref{eq:defQ} satisfies
  \begin{equation*}
    \neutrino (Q) = A \neutrino (P).
  \end{equation*}
  For $\lambda\in\C$, we infer that
  \begin{eqnarray*}
    (A-\lambda I)\neutrino (P)  & = & A \neutrino (P) - \lambda \neutrino (P)\\
                                & = & \neutrino (Q) - \lambda V_c D \mathfrak c^1(Q)\\
                                & = & V \mathfrak c^1(Q) - \lambda V_c D \mathfrak c^1(Q)\\
                                & = & \left(V- \lambda V_c D \right)\mathfrak c^1(Q)\\
    & = & \left(V- \lambda V_c D \right) D^{-1} V_c^{-1} \neutrino (P).
  \end{eqnarray*}
  Since this holds for all $P$, we infer that
  \begin{equation}
    \label{eq:chiAgene}
    \chi_A(\lambda)={\rm det}(\lambda I-A) = {\rm det} (\lambda V_c D-V)
    \times {\rm det} (D^{-1} V_c^{-1}).
  \end{equation}
  Observe that, setting
  \begin{equation}
    \label{eq:defR1N}
    R_1
    =
    \begin{pmatrix}
      1 & 0 & \cdots & 0 \\
      \vdots & \vdots &   & \vdots\\
      1 & 0 & \cdots & 0
    \end{pmatrix}
    \qquad
    \text{and}
    \qquad
    N
    =
      \begin{pmatrix}
      0 & 2 & 0 & \cdots & 0 \\
      0 & 0 & 3 & \ddots & \vdots \\
      \vdots & \ddots & \ddots & \ddots & 0 \\
      \vdots &    & \ddots & 0 & s\\
      0 & \cdots & \cdots & 0 & 0
    \end{pmatrix}
    ,
  \end{equation}
  we have
  \begin{equation*}
    V_c D = VN + R_1.
  \end{equation*}
  Hence, using \eqref{eq:chiAgene}, $\chi_A$ is proportional to
  ${\rm det} (\lambda V N - V + \lambda R_1)$.
  Introducing the invertible matrix $M_\lambda=-I+\lambda N$, observe that
  \begin{equation*}
    \lambda V N - V + \lambda R_1 = V (-I + \lambda N + \lambda V^{-1}R_1)
    = V (I+ \lambda V^{-1}R_1M_\lambda^{-1}) M_\lambda.
  \end{equation*}
  Since ${\rm det}V$ and ${\rm det} M_\lambda=(-1)^s$ do not depend on $\lambda$, we infer that
  $\chi_A$ is proportional to ${\rm det}(I+ \lambda V^{-1}R_1M_\lambda^{-1})$.
  Observe also that since we have the Dunford decomposition of $M_\lambda$, we have
  \begin{equation*}
    (-M_\lambda)^{-1} = (I-\lambda N)^{-1} = \sum_{k=0}^{s-1} (\lambda N)^k.
  \end{equation*}
  Since for all $k\in\{1,\cdots,s-1\}$, the matrix $N^k=(n_{i,j}^k)_{1\leq i,j\leq s}$
  is a matrix with coefficients zero except for
  $(n^k)_{i,i+k}= \frac{(k+i)!}{(i)!}$ for $1\leq i\leq s-k$, we infer that the
  first row of $M_{\lambda}^{-1}$ is equal to
  $(-1!,-2!\lambda,\cdots,-(s-1)!\lambda^{s-2},-s!\lambda^{s-1})$.
  This implies that $R_1 M_\lambda^{-1}$ is a matrix with all rows equal to
  $(-1!,-2!\lambda,\cdots,-(s-1)!\lambda^{s-2},-s!\lambda^{s-1})$.
  This is equivalent to saying that for all $j\in\{1,\cdots,s\}$, the column number $j$ of
  the matrix $R_1 M_\lambda^{-1}$ is equal to $-(j!)\lambda^{j-1} {\mathds 1}$.
  This implies that the column number $j\in\{1,\hdots,s\}$ of the matrix $V^{-1}R_1M_\lambda^{-1}$ reads
  $-j !\lambda^{j-1} V^{-1} \mathds{1}$.
  Since the matrix $V^{-1} R_1 M_\lambda^{-1}$ has rank 1 due to the form of $R_1$,
  there exists some invertible matrix $B$ and some nontrivial $(\xi_1,\dots,\xi_s)\in\C^s$,
  \begin{equation*}
    B^{-1} (V^{-1} R_1 M_\lambda^{-1}) B =
    \begin{pmatrix}
      \xi_1 & 0 & \cdots & 0 \\
      \xi_2 & 0 & \ddots & \vdots\\
      \vdots & \vdots  & \ddots & 0 \\
      \xi_s & 0 & \cdots & 0
    \end{pmatrix}:=\Xi.
  \end{equation*}
  We infer that
  \begin{equation*}
    \det(I + \lambda V^{-1} R_1 M_\lambda^{-1}) =\det(B(I+\lambda \Xi)B^{-1}) = (1+\lambda \xi_1).
  \end{equation*}
  Observe that $\xi_1={\rm Tr}(\Xi)={\rm Tr}(V^{-1} R_1 M_\lambda^{-1})$.
  In order to compute $\xi_1={\rm Tr}(V^{-1}R_1M_\lambda^{-1})$,
  we compute the entry number $i$ of $-i !\lambda^{-1} V^{-1} \mathds{1}$ for $i\in\{1,\cdots,s\}$.
  The unknown vector $X=(x_1,\cdots,x_s)\in\C^s$ satisfies $VX=\mathds{1}$ if and only
  if the polynomial $T(X)=x_s X^s + \cdots + x_1 X^1$ satisfies for all $j\in\{1,\dots s\}$,
  $T(c_j)=1$. Therefore, the polynomial $T(X)-1\in\C_s[X]$ is a multiple of
  $\pi=\prod_{i=1}^s(X-c_i)$ and therefore writes
  \begin{equation*}
    T-1=e\prod_{i=1}^s(X-c_i),
  \end{equation*}
  for some $e\in\C$. Observe that $T(0)=0$ hence, $-1= e \prod_{i=1}^s(-c_i)$.
  Hence, $e\neq 0$.
  Since,
  \begin{eqnarray*}
    \lambda \xi_1 & = & \lambda {\rm Tr}(V^{-1} R_1 M_\lambda^{-1}) \\
    & = &  - \sum_{i=1}^s i ! \lambda^i (V^{-1}\mathds{1})_{i}\\
    & = & - \sum_{i=1}^s i! x_i\lambda^i \\
    & = & \tau(-T)(\lambda),
  \end{eqnarray*}
  we conclude that
  \begin{eqnarray*}
    1+\lambda\xi_1 & = & \tau(1-T)(\lambda)\\
            & = & - e \tau(\pi)(\lambda).
  \end{eqnarray*}
  Since $e\neq 0$, we conclude that $\chi_A$ is proportional to $\tau(\pi)$.
  Analysing the coefficient of degree $s$ ensures that $\chi_A=\tau(\pi)/(s!)$.
  This concludes the proof when all $c_i\neq 0$.
\end{proof}

\begin{remark}
  \label{rem:0vpA}
  With the hypotheses and notations above, $0$ is an eigenvalue of $A$ if and only if
  there exists $j\in\{1,\cdots,s\}$ such that $c_j=0$.
  This is also equivalent to $0$ being a root of $\tau(\pi)$.
\end{remark}

\begin{remark}
  Using Lemma \ref{lem:transfoTau}, the eigenvalues of $A$ are the zeros of $\tau(\pi)$.
  Observe also that
  \begin{equation}
    \label{eq:stabCmoins}
    \left\{ \frac{1}{z}\ |\ z\in\C^-\setminus\{0\} \right\} = \C^{-}\setminus\{0\}.
  \end{equation}
  Therefore, in view of Definition \ref{def:stab}, a necessary condition
  to have $ASI$ stability is that $\tau(\pi)$ has no root in $\C^- \setminus\{0\}$.
  Similarly, observing that
  \begin{equation}
    \label{eq:stabiRmoins}
    \left\{ \frac{1}{z}\ |\ z\in i \R\setminus\{0\} \right\} = i \R\setminus\{0\},
  \end{equation}
  we obtain that a necessary
  condition to have $ISI$ stability is that $\tau(\pi)$ has no root in $i\R\setminus\{0\}$.
\end{remark}

\subsection{Sufficient conditions to ensure $\hat A$-stability and $\hat I$-stability}

We prove that an $A$-stable Runge--Kutta collocation method such that $A$
has no eigenvalue in $\C^-$ is $\hat{A}$-stable.
Thanks to Lemma \ref{lem:transfoTau}, this is equivalent to proving that
an $A$-stable Runge--Kutta collocation method such that $\tau(\pi)$
has no root in $\C^-$ is $\hat{A}$-stable.
\begin{lemma}
  \label{lem:AdonneAchapeau}
  Assume that an $A$-stable Runge--Kutta collocation method is such that $\tau(\pi)$ has
  no root in $\C^-$. Then, it is $\hat A$-stable.
\end{lemma}

\begin{proof}
  Let $c_1<\cdots<c_s$ be such that the Runge--Kutta collocation method with matrix $A$ is $A$-stable.
  Observe that 
  \begin{equation*}
    \forall \lambda\in\C\setminus\{0\},\qquad
    \left(I-\lambda A \right)\text{ is invertible} \Longleftrightarrow
     \left(A-\frac{1}{\lambda} I\right) \text{ is invertible}.
  \end{equation*}
  Since $\tau(\pi)$ has no zero in $\C^-$ by hypothesis, Lemma \ref{lem:transfoTau} ensures that
  the matrix $A$ has no eigenvalue in $\C^-$.
  In particular, for all $\lambda\in\mathcal \C^-\setminus\{0\}$, $(A-\lambda I)$ is invertible.
  Observing that
  \begin{equation}
    \label{eq:identitespectrale}
    \forall \lambda\in\C\setminus\{0\},\qquad
    (I-\lambda A)^{-1} = -\frac{1}{\lambda}\left(A-\frac{1}{\lambda}I\right)^{-1}.
  \end{equation}
  This identity implies that we have
  \begin{equation}
    \label{eq:letruccool}
     \lambda\mapsto\left(A-\lambda I\right)^{-1} \text{ is bounded over $\C^-\setminus\{0\}$}
     \quad \Longrightarrow\quad 
    \lambda\mapsto \left(I-\lambda A \right)^{-1}\text{ is bounded over $\C^-\setminus\{0\}$}.
  \end{equation} 
  Indeed, $\lambda \mapsto (I-\lambda A)^{-1}$ is always bounded in a complex neighbourhood of $0$.

  Observe that $(A-\lambda I)^{-1}$ is a matrix, the coefficients of which are rational functions
  of $\lambda$ of negative degrees. Therefore, since $A$ has no eigenvalue in $\C^-$, we infer
  that $\lambda \mapsto (A-\lambda I)^{-1}$ is bounded on $\C^-$.
  Therefore, using \eqref{eq:letruccool}, we obtain that
  $\lambda \mapsto (I-\lambda A)^{-1}$ in bounded on $\C^-$.
  This implies that the method is $ASI$-stable.
  Since $A$ has no eigenvalue in $\C^-$, we infer that $0$ is not an eigenvalue of $A$.
  Hence, using Lemma \ref{lem:transfoTau}, the matrix $A$ is invertible.
  Using Remark \ref{rem:SIimpliqueS}, we infer that the method is also $AS$-stable using .
  Finally, the method is $A$, $AS$ and $ASI$-stable. Hence, it is $\hat A$-stable.
\end{proof}

A direct adaptation of the previous result yields the following lemma.
\begin{lemma}
  \label{lem:IdonneIchapeau}
  Assume that an $I$-stable Runge--Kutta collocation method is such that $\tau(\pi)$ has
  no root in $i\R$. Then, it is $\hat I$-stable.
\end{lemma}

\subsection{Stability of Runge--Kutta collocation methods with Gauss points}

We recall this result, established in \cite{Ehle68} and \cite{Wright70} :
\begin{proposition}
  \label{prop:GaussAstable}
  An $s$ stages collocation Runge--Kutta method with $s$ Gauss points is $A$-stable.
\end{proposition}

\begin{proof}
  Let us simply make the link between the proof of \cite{Wright70} and the notations introduced
  above.
  The rational stability function $R$ introduced in \eqref{eq:defR} satisfies for $\Re(\lambda)>0$
  \begin{equation}
    \label{eq:LaplaceR}
    R(\lambda) = \frac{\displaystyle\int_0^{+\infty}{\rm e}^{-\lambda\sigma}\pi(\sigma+1)\dd \sigma}{\displaystyle\int_0^{+\infty}{\rm e}^{-\lambda\sigma}\pi(\sigma)\dd \sigma}.
  \end{equation}
  Using the fact that for all $k\in\N$,
  $\int_0^{+\infty}{\rm e}^{-\lambda\sigma}\sigma^k\dd\sigma=k!/\lambda^{k+1}$, we infer that
  \begin{equation}
    \label{eq:Rtaudepi}
    R(\lambda)=\frac{\tau(\pi(X+1))(1/\lambda)}{\tau(\pi(X))(1/\lambda)},
  \end{equation}
  and this relation also holds for all $\lambda\in\C\setminus\{0\}$ such that $1/\lambda$ is not
  a zero of $\tau(\pi)$.
  When the $(c_i)_{1\leq i\leq s}$ are Gauss points, one has
  \begin{equation}
    \label{eq:pideXGauss}
    \pi(X) = \frac{\dd^s}{\dd X^s} (X(X-1))^s.    
  \end{equation}
  Following \cite{Wright70}, we apply the Routh--Hurwitz
  algorithm to $\tau(\pi)$ in order to show that its roots have positive real parts.
  By a simple computation from \eqref{eq:pideXGauss}, we have
  \begin{eqnarray*}
    \pi(X) & = & \frac{{\rm d}^s}{{\rm d}X^s} \left(X^s \sum_{k=0}^s \frac{s!}{(k!(s-k)!)}(-X)^k\right)\\
    & = & \sum_{k=0}^s \frac{s!}{k!(s-k)!} \frac{(s+k)!}{k!}(-X)^k.
  \end{eqnarray*}
  With the definition of $\tau(\pi)$ from Lemma \ref{lem:transfoTau}, we infer that
  \begin{equation*}
    \tau(\pi)(\lambda) = s! \sum_{k=0}^s \frac{(s+k)!}{k!(s-k)!} (-\lambda)^k.
  \end{equation*}
  This corresponds to the polynomial in the appendix of \cite{Wright70} with $-\lambda$ here
  instead of $\mu$ there.
  In particular, using \cite{Wright70}, we infer that all the roots of $\tau(\pi)$ have positive
  real parts.
  
  Using \eqref{eq:Rtaudepi}, we infer that $R(\lambda)$ has no pole in $\C^-$.
  Therefore, $\lambda\mapsto R(\lambda)$ is an analytic function in an open set containing $\C^-$.
  Following \cite{Wright70}, we have that $R(\lambda)$ has modulus bounded by one
  on the purely imaginary axis as well as on the left half circle at infinity
  (see also Section \ref{subsec:symmetry}).
  Using the maximum principle for the analytic function $R$, we infer that
  $|R(\lambda)|\leq 1$ in the closed left half plane $\C^-$.
  Therefore, the collocation method with Gauss points is $A$-stable.
\end{proof}

\begin{remark}
  Relations \eqref{eq:LaplaceR} (for $\Re(\lambda)>0$) and \eqref{eq:Rtaudepi}
  (for $\lambda\in\C\setminus\{0\}$) hold for all collocation methods.
  We will use this fact in Section \ref{subsec:algo} to derive an algorithm
  that allows to ensure that a given collocation method is $\hat A$-stable.
\end{remark}

Using Lemma \ref{lem:AdonneAchapeau}, this result can be improved as follows.
\begin{theorem}
  \label{th:GaussAchapeau}
  An $s$ stages collocation Runge--Kutta method with $s$ Gauss points is $\hat A$-stable and
  $\hat I$-stable.
\end{theorem}

\begin{proof}
  Consider an $s$ stages collocation Runge--Kutta method with Gauss points.
  That method is $A$-stable using Proposition \ref{prop:GaussAstable}, and the roots of $\tau(\pi)$
  have positive real parts, as we have seen in the proof of Proposition \ref{prop:GaussAstable}.
  Using Lemma \ref{lem:AdonneAchapeau}, we infer that the method is $\hat A$-stable.
  It is also $\hat I$-stable by Remark \ref{rem:AdonneI}.
\end{proof}

\subsection{An algorithm to ensure that a given
  collocation method with $s$-stages is $\hat A$-stable (hence $\hat I$-stable)}
\label{subsec:algo}

If a collocation method is given by $ c_1<c_2<\hdots<c_s $, then one can form
the polynomial $\pi(X)=\prod_{i=1}^s(X-c_i)$ and the polynomial $\tau(\pi)$ using the definition of
$\tau$ in Lemma \ref{lem:transfoTau}. Using the fact that $\lambda\mapsto R(\lambda)$
is a rational function, the maximum principle implies that the method is $A$-stable if and only if
\begin{itemize}
\item $\lambda\mapsto R(\lambda)$ has no pole in $\C^-$,
\item $|R(\lambda)|\leq 1$ for $\lambda\in i\R$ (and hence in the limit $|\lambda|\mapsto +\infty$ in $\C^-$).
\end{itemize}

These two assumptions can be examinated separately using algorithm that allow to localize
the roots of polynomials as follows:
\begin{itemize}
\item For the first point above,
if we assume that the collocation method is such that $\tau(\pi)$ has no root in $\C^-$,
then, using \eqref{eq:Rtaudepi}, this first point holds.
Moreover, this hypothesis on $\tau(\pi)$ can be checked using the Routh--Hurwitz algorithm
\cite{RH1895}.
This is for example the case for Gauss points (see the proof of Theorem \ref{th:GaussAchapeau}).
\item The second point can be checked by examining the real roots of $|R(ix)|^2-1$.
\end{itemize}

This provides sufficient conditions to decide whether a given Runge--Kutta collocation method
is $A$-stable or not.

For a Runge--Kutta collocation method satisfying the two points above, we have that it is $A$-stable.
Moreover, since $\tau(\pi)$ has no pole in $\C^-$, the method satisfies that all $(c_i)_{1\leq i\leq s}$
are non-zero and the matrix $A$ is invertible..
In this case, using Lemma \ref{lem:AdonneAchapeau}, we conclude that the method is $\hat{A}$-stable.

However, a method may be $A$-stable while $\tau(\pi)$ has zeros in $\C^-$,
and that method may fail to be $\hat{A}$-stable
(see Example \ref{ex:5etagesAASpasASI} below).

The study of a few examples of Runge--Kutta methods can be found in Section \ref{sec:examples}.

\subsection{Symmetry of the $(c_i)_{1\leq i\leq s}$ with respect to $1/2$ and stability}
\label{subsec:symmetry}

\begin{proposition}
  \label{prop:symmetry}
  Assume the $s$ distinct coefficients $c_1<\cdots<c_s$ of a Runge--Kutta collocation
  method are symmetric with respect to $1/2$, {\it i.e.}
  for all $i\in\{1,\cdots,s\}$, $c_i+c_{s+1-i}=1$.
  The linear stability function $R$ of this method satisfies
  \begin{equation}
    \label{eq:Ralinfini}
    \lim_{|\lambda| \to +\infty} |R(\lambda)|\leq 1,
  \end{equation}
  and
  \begin{equation}
    \label{eq:Rsurir}
    \forall i\in {\rm i}\R,\qquad |R(\lambda)|=1.
  \end{equation}
\end{proposition}

\begin{proof}
  The result of Section \ref{subsec:algo} above holds independantly of any symmetry condition
  on the collocation points $c_1,\hdots,c_s$.
  However, when the points $c_1,\cdots,c_s$ are symetric with respect to $1/2$, we have
  \begin{equation*}
    \pi(X+1) = \prod_{i=1}^s(X+1-c_i) =  \prod_{i=1}^s (X+c_i) = (-1)^s \prod_{i=1}^s (-X-c_i)
    = (-1)^s \pi (-X).
  \end{equation*}
  This implies that
  \begin{equation}
    \label{eq:identitesymetrique}
    \tau(\pi(X+1))(\lambda)=(-1)^s\tau(\pi(X))(-\lambda).
  \end{equation}
  Hence, using \eqref{eq:Rtaudepi}, we infer
  \begin{equation*}
    |R(\lambda)| =
    \left|\frac{\tau(\pi(X+1))(1/\lambda)}{\tau(\pi(X))(1/\lambda)}\right|
    =
    \left|\frac{\tau(\pi(X+1))(1/\lambda)}{\tau(\pi(X+1))(-1/\lambda)}\right|.
  \end{equation*}
  This implies \eqref{eq:Ralinfini}, as already noted by \cite{Wright70}.
  Expanding the polynomial with real coefficients $\tau(\pi(X))$
  and writing it as $\tau(\pi(X))(\lambda)=\sum_{k=0}^s \alpha_k \lambda^k$,
  Relation \eqref{eq:identitesymetrique} also implies that, for $x\in\R$,
  \begin{eqnarray*}
    \tau(\pi(X+1))(\icomplexe x) & = & (-1)^s\tau(\pi(X))(-\icomplexe x)\\
                        & = & (-1)^s \left(\sum_{\substack{k=0\\ k\ {\rm even}}}^s \alpha_k (-\icomplexe x)^k
    +  \sum_{\substack{k=0\\ k\ {\rm odd}}}^s \alpha_k (-\icomplexe x)^k\right)\\
    & = & (-1)^s\left(\sum_{\substack{k=0\\ k\ {\rm even}}}^s \alpha_k (\icomplexe x)^k
    - \sum_{\substack{k=0\\ k\ {\rm odd}}}^s \alpha_k (\icomplexe x)^k\right)\\
    & = & (-1)^s\left(\sum_{\substack{k=0\\ k\ {\rm even}}}^s \alpha_k \overline{(\icomplexe x)^k}
    +  \sum_{\substack{k=0\\ k\ {\rm odd}}}^s \alpha_k \overline{(\icomplexe x)^k}\right)\\
    & = & (-1)^s \overline{\tau(\pi(X))(\icomplexe x)}.
  \end{eqnarray*}
  In view of \eqref{eq:Rtaudepi}, this implies \eqref{eq:Rsurir}.
\end{proof}

\begin{remark}
  \label{rem:CS-symmIstable}
  Using Proposition \ref{prop:symmetry}, a Runge--Kutta collocation method
  with symmetric points is $I$-stable.
\end{remark}

\begin{remark}
  \label{rem:CS-Ichapeaustable}
  A Runge--Kutta collocation method with symmetric points
  such that $\tau(\pi)$ has no roots in ${\rm i}\R$ is $\hat{I}$-stable.
  Indeed, using the remark above, such a method is $I$-stable.
  With Lemma \ref{lem:IdonneIchapeau}, we infer that it is $\hat I$-stable.
\end{remark}

\begin{remark}
  \label{rem:CS-Achapeaustable}
  Observe that a Runge--Kutta collocation method
  with symmetric points that is such that $\tau(\pi)$ has no poles in $\C^-$
  is $\hat A$-stable.
  Indeed, using Proposition \ref{prop:symmetry}, the associated linear stability function $R$
  satisfies \eqref{eq:Ralinfini} and \eqref{eq:Rsurir}.
  Moreover, the poles of $R$ in $\C^-$ are inverses of eigenvalues of $A$ in $\C^-\setminus\{0\}$.
  Since $\tau(\pi)$ has no root in $\C^-$, Lemma \ref{lem:transfoTau} ensures that
  $A$ has no such eigenvalue. In particular, $R$ is analytic in an open set containing $\C^-$.
  The maximum principle ensures that $|R(\lambda)|\leq 1$ over $\C^-$.
  Hence, the method is $A$-stable.
  Using Lemma \ref{lem:AdonneAchapeau}, we infer that it actually is $\hat A$-stable.
\end{remark}

\subsection{$I$- and $A$-stability of forward
  collocation Runge--Kutta methods with at most $4$ stages}
\label{subsec:forward1234}

\subsubsection{Statement of the result}
This section is devoted to the proof of the following theorem.
The proof of this theorem relies on lemmas allowing to locate the eigenvalues of the matrix $A$
for forward Runge--Kutta collocation methods with at most $4$ stages.
The statements and proofs of these lemmas are provided in Section \ref{subsubsec:localization1234}.
These lemmas rely on the Routh--Hurwitz criterion for polynomials of degrees at most $4$,
which is recalled in Section \ref{subsubsec:RH34}.

\begin{theorem}\label{th:1234stagesIimpliqueA}
  Consider a forward collocation Runge--Kutta method with $s=1$, $s=2$, $s=3$ or $s=4$ stages.
  If it is $I$-stable, then it is $A$-stable.
\end{theorem}

\begin{remark}
  This result shows that we can not construct an $I$-stable with $1$, $2$, $3$ or $4$ stages
  at non-negative collocation points that is both $I$-stable and {\it not} $A$-stable.
  In contrast,  Example \ref{ex:IstablemaispasAstable} shows that it is possible
  to construct $I$-stable yet not $A$-stable forward Runge--Kutta collocation methods
  with $s=5$ stages.
\end{remark}

\begin{proof}
  Assume $s\in\{1,2,3,4\}$ and the collocation points $(c_i)_{1\leq i\leq s}$ are non-negative.
  Using either Lemma \ref{lem:spectreA1etage} (if $s=1$), \ref{lem:spectreA2etages} (if $s=2$),
  \ref{lem:spectreA3etages} (if $s=3$) or \ref{lem:spectreA4etages} (if $s=4$),
  we infer that ${\rm sp}(A)\subset \{0\} \cup\{z\in\C\ |\ \Re(z)>0\}$.
  In particular, the function $\lambda\mapsto (I-\lambda A)^{-1}$ is analytic over
  $\{z\in\C\ |\ \Re(z)<0\}$. This implies that $\lambda\mapsto R(\lambda)$ (see \eqref{eq:defR})
  is analytic over the same open set.
  Assume now that the collocation Runge--Kutta method is $I$-stable (see Definition \ref{def:stab}).
  This means that
  \begin{equation*}
    \forall x\in\R,\qquad |R(ix)|\leq 1.
  \end{equation*}
  In particular, the rational function $R$ has no pole in $\{z\in\C \ |\ \Re(z)\leq 0\}$.
  Using the maximum principle, we infer that $\lambda \mapsto |R(\lambda)|$ is bounded by $1$
  over this set $\{z\in\C \ |\ \Re(z)\leq 0\}$. Hence, the method is $A$-stable.
\end{proof}

\subsubsection{Eigenvalues of a forward collocation Runge--Kutta method with
  $s\in\{1,2,3,4\}$ stages}
\label{subsubsec:localization1234}

\begin{lemma}
  \label{lem:spectreA1etage}
  If $s=1$ and $0\leq c_1$ is a given real number,
  then the corresponding $1$-stage Runge--Kutta collocation method
  with matrix $A=(a_{i,j})_{1\leq i,j\leq 3}$ satisfies
  \begin{equation*}
    {\rm sp} (A) = \{c_1\} \subset \left\{z\in\C\ |\ \Re(z) \geq 0\right\}.
  \end{equation*}
\end{lemma}

\begin{proof}
  This is obvious using for example relation \eqref{eq:relationA} of Lemma \ref{lem:relationA}.
\end{proof}

\begin{lemma}
  \label{lem:spectreA2etages}
  If $s=2$ and $0< c_1<c_2$ are two real numbers,
  then the corresponding $2$-stages Runge--Kutta collocation method
  with matrix $A=(a_{i,j})_{1\leq i,j\leq 3}$ satisfies
  \begin{equation*}
    {\rm sp} (A) \subset \left\{z\in\C\ |\ \Re(z) > 0\right\}.
  \end{equation*}
  If $s=2$ and $0= c_1<c_2$ are two real numbers,
  then the corresponding $3$-stages Runge--Kutta collocation method
  with matrix $A=(a_{i,j})_{1\leq i,j\leq 3}$ satisfies
  \begin{equation*}
    {\rm sp} (A) \subset \{0\} \cup \left\{z\in\C\ |\ \Re(z) > 0\right\}.
  \end{equation*}
\end{lemma}

\begin{proof}
  Using for example relation \eqref{eq:relationA} of Lemma \ref{lem:relationA}, we have
  \begin{equation*}
    \chi_A(X) =  X^2 - \frac{c_1+c_2}{2} X + \frac{c_1c_2}{2}.
  \end{equation*}
  If $c_1=0$ then the roots of $\chi_A$ are $0$ and $c_2/2>0$.
  If $c_1>0$, then $c_2>0$ and the roots of $\chi_A$ have positive real parts.
\end{proof}

\begin{lemma}
  \label{lem:spectreA3etages}
  If $s=3$ and $0< c_1<c_2<c_3$ are three real numbers,
  then the corresponding $3$-stages Runge--Kutta collocation method
  with matrix $A=(a_{i,j})_{1\leq i,j\leq 3}$ satisfies
  \begin{equation*}
    {\rm sp} (A) \subset \left\{z\in\C\ |\ \Re(z) > 0\right\}.
  \end{equation*}
  If $s=3$ and $0= c_1<c_2<c_3$ are three real numbers,
  then the corresponding $3$-stages Runge--Kutta collocation method
  with matrix $A=(a_{i,j})_{1\leq i,j\leq 3}$ satisfies
  \begin{equation*}
    {\rm sp} (A) \subset \{0\} \cup \left\{z\in\C\ |\ \Re(z) > 0\right\}.
  \end{equation*}
\end{lemma}

\begin{proof}
  Using \eqref{eq:relationA} of Lemma \ref{lem:relationA}, a direct computation
  ensures that the characteristic polynomial of $A$ reads
\begin{equation}
\label{eq:caracpoly_A}
\chi_A(X)=X^3-\dfrac{1}{3}(c_1+c_2+c_3)X^2+\dfrac{1}{6}(c_1c_3+c_2c_3+c_1c_2)X-\dfrac{1}{6}c_1c_2c_3.
\end{equation}
In the first case $c_1>0$ (hence $c_3>c_2>c_1>0$), the Routh-Hurwitz criterion
(Proposition \ref{prop:RH3}) applied to the third order polynomial $\chi_A(-X)$ ensures that
if
\begin{equation}
  \label{eq:RH3}
  \dfrac{1}{3}(c_1+c_2+c_3) \dfrac{1}{6}(c_1c_3+c_2c_3+c_1c_2) - \dfrac{1}{6}c_1c_2c_3>0
\end{equation}
then all the roots of $\chi_A$ have positive real parts. Moreover, relation \eqref{eq:RH3} is obvious
by direct expansion of the left hand side using the fact that for all $i\in\{1,\cdots,3\}$, $c_i>0$.

In the second case $c_1=0<c_2<c_3$, we have
\begin{equation}
\label{eq:caracpoly_A2}
\chi_A(X)=X\left( X^2-\dfrac{c_2+c_3}{3}X+\dfrac{c_2c_3}{6} \right).
\end{equation}
This ensures that all the roots of $\chi_A(X)/X$ have positive real parts.
\end{proof}

\begin{lemma}
  \label{lem:spectreA4etages}
  If $s=4$ and $0< c_1<c_2<c_3<c_4$ are four real numbers,
  then the corresponding $4$-stages Runge--Kutta collocation method
  with matrix $A=(a_{i,j})_{1\leq i,j\leq 4}$ satisfies
  \begin{equation*}
    {\rm sp} (A) \subset \left\{z\in\C\ |\ \Re(z) > 0\right\}.
  \end{equation*}
    If $s=4$ and $0= c_1<c_2<c_3<c_4$ are four real numbers,
  then the corresponding $4$-stages Runge--Kutta collocation method
  with matrix $A=(a_{i,j})_{1\leq i,j\leq 4}$ satisfies
  \begin{equation*}
    {\rm sp} (A) \subset \{0\} \cup \left\{z\in\C\ |\ \Re(z)> 0\right\}.
  \end{equation*}
\end{lemma}

\begin{proof}
  Using \eqref{eq:relationA} of Lemma \ref{lem:relationA}, a direct computation
  ensures that the characteristic polynomial of $A$ reads
\begin{equation}
\label{eq:caracpoly_A4}
\chi_A(X)=X^4-\dfrac{1}{4}\Sigma_1 X^3+\dfrac{1}{12}\Sigma_2X^2-\dfrac{1}{24}\Sigma_3 X+\dfrac{1}{24}\Sigma_4,
\end{equation}
where
\begin{equation}
  \label{eq:defSigma4}
\Sigma_1=\displaystyle{\sum_{i=1}^4 c_i},\qquad
\Sigma_2=\displaystyle{\sum_{1\leq i<j \leq 4}c_ic_j},\qquad
\Sigma_3=\displaystyle{\sum_{1\leq i<j<k \leq 4}c_ic_jc_k},\qquad
\Sigma_4=c_1c_2c_3c_4,
\end{equation}
are the elementary symmetric polynomials in $c_1, c_2, c_3, c_4$.
Assume first that $c_1>0$. In this case, we have also $c_2,c_3,c_4>0$.
In particular, with \eqref{eq:defSigma4}, we have for all $i\in\{1,\cdots,4\}$, $\Sigma_i>0$.
Therefore, by the Routh-Hurwitz criterion (Proposition \ref{prop:RH4-alt} applied to $\chi_A(-X)$), if
\begin{equation}
  \label{eq:RH4}
  \frac{\Sigma_1}{4} \frac{\Sigma_2}{12} \frac{\Sigma_3}{24}>\left(\frac{\Sigma_3}{24}\right)^2
  +\left(\frac{\Sigma_1}{4}\right)^2 \frac{\Sigma_4}{24},
\end{equation}
then the fourth order polynomial $\chi_A$ has all its roots in the open right half plane
$\{z\in\C\ |\ \Re(z)>0\}$.
Observe that, in our case,
\begin{eqnarray*}
  \lefteqn{\Sigma_1 \Sigma_2 \Sigma_3 -2 \left(\Sigma_3\right)^2
  -3\left(\Sigma_1\right)^2 \Sigma_4 = }\\
   & & \begin{matrix}
          + c_1^3 c_2^2 c_3  & +c_1^3 c_2^2 c_4 & +c_1^3 c_2 c_3^2 & +c_1^3 c_2 c_4^2 & +c_1^3 c_3^2 c_4 & +c_1^3 c_3 c_4^2 & +c_1^2 c_2^3 c_3\\
          +c_1^2 c_2^3 c_4 & {\color{blue}+c_1^2 c_2^2 c_3^2} & {\color{blue}-2 c_1^2 c_2^2 c_3 c_4} & {\color{blue}+c_1^2 c_2^2 c_4^2} & +c_1^2 c_2 c_3^3 & {\color{orange}-2 c_1^2 c_2 c_3^2 c_4} & {\color{green}-2 c_1^2 c_2 c_3 c_4^2} \\
          +c_1^2 c_2 c_4^3 & {\color{orange}+c_1^2 c_3^3 c_4}  & {\color{orange}+c_1^2 c_3^2 c_4^2} & {\color{green}+c_1^2 c_3 c_4^3} & +c_1 c_2^3 c_3^2  & +c_1 c_2^3 c_4^2 & +c_1 c_2^2 c_3^3\\
          {\color{purple}-2 c_1 c_2^2 c_3^2 c_4} & {\color{red}-2 c_1 c_2^2 c_3 c_4^2} & {\color{green}+c_1 c_2^2 c_4^3}
          & {\color{brown}-2 c_1 c_2 c_3^2 c_4^2} & {\color{purple}+c_1 c_3^3 c_4^2} & {\color{purple}+c_1 c_3^2 c_4^3} & + c_2^3 c_3^2 c_4\\
          {\color{red}+c_2^3 c_3 c_4^2} &  +c_2^2 c_3^3 c_4  & {\color{red}+c_2^2 c_3^2 c_4^2} & +c_2^2 c_3 c_4^3  & {\color{brown}+c_2 c_3^3 c_4^2} & {\color{brown}+c_2 c_3^2 c_4^3} & 
        \end{matrix} .
\end{eqnarray*}
This symmetric polynomial is {\it not} non-negative over $\R^4$.
However, it is positive over $(0,+\infty)^4$. Indeed, recalling that $0<c_1<c_2<c_3<c_4$,
we have that the sum of the 3 blue (respectively green, yellow, purple, orange and brown) terms is
positive. For example, for the orange terms,
\begin{equation*}
  2 c_1^2 c_2 c_3^2 c_4 = c_1^2 c_2 c_3^2 c_4 + c_1^2 c_2 c_3^2 c_4 < c_1^2 c_3^3 c_4 + c_1^2c_3^2c_4^2.
\end{equation*}
Hence, $\Sigma_1 \Sigma_2 \Sigma_3 -2 \left(\Sigma_3\right)^2 -3\left(\Sigma_1\right)^2 \Sigma_4 >0$.
Using the Routh-Hurwitz criterion (Proposition \ref{prop:RH4-alt}), we infer that all the roots of $\chi_A$ have positive
real parts in this case.

\noindent
Assume now that $c_1=0$. In this case, we have $0<c_2<c_3<c_4$. Moreover,
\begin{equation*}
\chi_A(X)=X\left(X^3-\dfrac{c_2+c_3+c_4}{4} X^2+\dfrac{(c_2c_3+c_2c_4+c_3c_4)}{12}X-\dfrac{c_2c_3c_4}{24}\right).
\end{equation*}
Using the Routh-Hurwitz criterion for the third order polynomial $\chi_A(-X)/X$
(Proposition \ref{prop:RH3}), we have that, if
\begin{equation}
  \label{eq:RH43}
  \dfrac{c_2+c_3+c_4}{4} \dfrac{(c_2c_3+c_2c_4+c_3c_4)}{12} - \dfrac{c_2c_3c_4}{24} >0,
\end{equation}
then all the roots of $\chi_A$ are in the closed right half plane $\{z\in\C\ |\ \Re(z)\geq 0\}$.
Moreover, \eqref{eq:RH43} is obvious by expanding the left hand side and using the positivity
of $c_2$, $c_3$ and $c_4$. This proves the result.
\end{proof}

\subsubsection{The Routh--Hurwitz stability criterion for polynomials of order $3$ and $4$}
\label{subsubsec:RH34}

The following two classical stability criterion are due to Routh and Hurwtiz \cite{RH1895}.

\begin{proposition}[Routh-Hurwitz criterion for polynomials of order 3]
  \label{prop:RH3}
  Assume that $P=X^3+a_2X^2+a_1X+a_0$ is a given polynomial of order $3$ with real coefficients.
  If $a_2>0$, $a_2a_1-a_0>0$ and $a_0>0$, then all the (complex) roots of $P$ have negative
  real parts.
\end{proposition}

\begin{proposition}[Routh-Hurwitz criterion for polynomials of order 4]
  Assume that $P=X^4+a_3X^3+a_2X^2 +a_1 X + a_0$ is a given polynomial of order $4$ with real
  coefficients.
  If $a_3>0$, $a_3a_2-a_1>0$, $a_3a_2a_1-a_1^2-a_3^2a_0>0$ and $a_0>0$, then all the (complex) roots of $P$ have negative
  real parts.
\end{proposition}

\begin{proposition}[alternative Routh-Hurwitz criterion for polynomials of order 4]
  \label{prop:RH4-alt}
  Assume that $P=X^4+a_3X^3+a_2X^2 +a_1 X + a_0$ is a given polynomial of order $4$ with real
  coefficients.
  If $a_3>0$, $a_1>0$, $a_3a_2a_1-a_1^2-a_3^2a_0>0$ and $a_0>0$, then all the (complex) roots of $P$ have negative
  real parts.
\end{proposition}


\section{Examples of Runge--Kutta collocation methods and analysis
  of their stability properties}
\label{sec:examples}

This section is devoted to the introduction of examples of forward Runge--Kutta collocation
methods with $s=2$, $s=4$ and $s=5$ stages and the analysis of their stability
properties with respect to Definition \ref{def:stab}.
Most of these methods are used as underlying methods in the linearly implicit methods
analyzed in \cite{DL2023PDE}.
In particular, some of them serve as examples in \cite{DL2023PDE} to illustrate the necessity
of all the stability criterions of Definition \ref{def:stab} for the convergence of the
linearly implicit methods described and implemented there, when applied to systems of ODEs
arising from the spatial discretization of evolution PDEs.

\subsection{Examples with $s=2$ stages}

We consider a collocation Runge--Kutta method with $s=2$ stages.
Assuming $c_1< c_2$, we have the Butcher tableau
  \[
\renewcommand\arraystretch{1.2}
\begin{array}
{c|cc}
 c_1 & \frac{c_1(c_1/2-c_2)}{c_1-c_2} & \frac{c_1^2/2}{c_1-c_2}\\[1 mm]
c_2 & \frac{c_2^2/2}{c_2-c_1} & \frac{c_2(c_2/2-c_1)}{c_2-c_1} \\ [1 mm]
\hline
& \frac{1/2-c_2}{c_1-c_2} & \frac{1/2-c_1}{c_2-c_1}
\end{array}.
\]
The stability function then reads
\begin{equation*}
  \begin {array}{c} R(\lambda) =  1+\lambda\, \left( -{\frac { \left( 1/2-{
\it c_2} \right)  \left( 2\,\lambda\,{\it c_1}\,{\it c_2}-\lambda\,{{\it 
c_2}}^{2}-2\,{\it c_1}+2\,{\it c_2} \right) }{ \left( {\it c_1}-{\it c_2}
 \right) ^{2} \left( {\it c_1}\,{\it c_2}\,{\lambda}^{2}-\lambda\,{\it 
c_1}-\lambda\,{\it c_2}+2 \right) }}-{\frac { \left( 1/2-{\it c_1}
 \right) \lambda\,{{\it c_2}}^{2}}{ \left( {\it c_2}-{\it c_1} \right) 
 \left( {\it c_1}-{\it c_2} \right)  \left( {\it c_1}\,{\it c_2}\,{\lambda
}^{2}-\lambda\,{\it c_1}-\lambda\,{\it c_2}+2 \right) }}\right.\\ \left.+{\frac {
 \left( 1/2-{\it c_2} \right) \lambda\,{{\it c_1}}^{2}}{ \left( {\it c_1}
-{\it c_2} \right) ^{2} \left( {\it c_1}\,{\it c_2}\,{\lambda}^{2}-
\lambda\,{\it c_1}-\lambda\,{\it c_2}+2 \right) }}-{\frac { \left( 1/2-{
\it c_1} \right)  \left( \lambda\,{{\it c_1}}^{2}-2\,\lambda\,{\it c_1}\,
{\it c_2}-2\,{\it c_1}+2\,{\it c_2} \right) }{ \left( {\it c_2}-{\it c_1}
 \right)  \left( {\it c_1}-{\it c_2} \right)  \left( {\it c_1}\,{\it c_2}
\,{\lambda}^{2}-\lambda\,{\it c_1}-\lambda\,{\it c_2}+2 \right) }}
 \right). \end {array}
\end{equation*}
These facts are used in the next four examples.

\begin{example}[Collocation Runge--Kutta method with $s=2$ Gauss points]
  \label{ex:meth2etagesGauss}
For Gauss points, the method reads
\begin{equation}
  \label{eq:Gauss2}
\renewcommand\arraystretch{1.2}
\begin{array}
{c|cc}
 \frac{1}{2}-\frac{\sqrt{3}}{6} & \frac{1}{4} & \frac{1}{4}-\frac{\sqrt{3}}{6}\\
\frac{1}{2}+\frac{\sqrt{3}}{6} & \frac{1}{4}+\frac{\sqrt{3}}{6} & 1/4\\
\hline
& \frac{1}{2} & \frac{1}{2} 
\end{array}.
\end{equation}
This method is $\hat A$-stable and hence $\hat I$-stable thanks to Theorem \ref{th:GaussAchapeau}.
\end{example}

\begin{example}[Collocation Runge--Kutta method with $s=2$ stages that is $\hat{A}$-stable]
  \label{ex:meth2etagesequi}
For the points $(c_1,c_2)=(1/3,2/3)$, the method reads,
\begin{equation}
  \label{eq:equispaced2}
\renewcommand\arraystretch{1.2}
\begin{array}
{c|cc}
 \frac13 & \frac12 & -\frac16\\
 \frac23 & \frac23 & 0\\[1mm]
\hline
& \frac{1}{2} & \frac{1}{2} 
\end{array}.
\end{equation}
One can check that the spectrum of the matrix $A$ consists in 2 complex conjugate eigenvalues
$\frac{1}{4}\pm{\rm i}\frac{\sqrt{7}}{12}$.
This method has symmetric points with respect to $1/2$ and $A$ has no eigenvalue in $\C^-$.
Using Remark \ref{rem:CS-Achapeaustable}, we infer that it is $\hat A$-stable, hence $\hat I$-stable.
\end{example}

\begin{example}[Lobatto collocation method with $s=2$ stages]
  \label{ex:meth2etagesLobatto}
  The Runge--Kutta collocation method with $s=2$ and $(c_1,c_2)=(0,1)$ reads
    \[
\renewcommand\arraystretch{1.2}
\begin{array}
{c|cc}
 0 & 0 & 0 \\[1 mm]
 1 & \frac12 & \frac12 \\ [1 mm]
\hline
& \frac12 & \frac12
\end{array}.
\]
  It is $A$-stable.
  Indeed, in this case,
  \begin{equation*}
    R(\lambda) = - \frac{\lambda+2}{\lambda-2},
  \end{equation*}
  and hence, for $\lambda\in\C^-$, $|R(\lambda)|\leq 1$.
  Moreover, this method is $ASI$-stable. Indeed, we have
  \begin{equation*}
    (I-\lambda A)^{-1}=
    \begin{pmatrix}
      1 & 0 \\
      \frac{\lambda}{2-\lambda} & \frac{2}{2-\lambda}
    \end{pmatrix},
  \end{equation*}
  and hence this matrix is bounded on $\C^-$. Finally, the method is $AS$-stable since
  \begin{equation*}
    \lambda b^t (I-\lambda A)^{-1} =
    \begin{pmatrix}
      \frac{\lambda}{2-\lambda} & \frac{\lambda}{2-\lambda}
    \end{pmatrix},
  \end{equation*}
  which is bounded on $\C^-$.
  Hence, this method is $\hat A$-stable.
\end{example}

\begin{example}[Collocation Runge--Kutta method with $s=2$ stages, not $A$- nor $I$-stable]
  \label{ex:notAstable}
The Butcher tableau of the Runge--Kutta collocation method with $s=2$ and $(c_1,c_2)=(1/4,1/3)$ reads
    \[
\renewcommand\arraystretch{1.2}
\begin{array}
{c|cc}
 \frac14 & \frac58 & -\frac38 \\[1 mm]
\frac13 & \frac23 & -\frac13 \\ [1 mm]
\hline
& -2 & 3
\end{array}.
\]
  It is {\bf neither} $A$-stable {\bf nor} $I$-stable.
  Indeed, in this case,
  \begin{equation*}
    R(\lambda) = \frac{6\lambda^2+17\lambda+24}{\lambda^2-7\lambda+24}.
  \end{equation*}
\end{example}

\subsection{Example with $s=4$ stages}

\begin{example}[Lobatto collocation method with $s=4$ uniform points]
  \label{ex:4etages}
  The Butcher tableau of the Runge--Kutta collocation method with $s=4$
  and $(c_1,c_2,c_3,c_4)=(0,\frac13,\frac23,1)$ reads:
\[
\renewcommand\arraystretch{1.2}
\begin{array}
{c|cccc}
0 & 0 & 0 & 0 & 0\\
1/3& 1/8 & 19/72 & -5/72 & 1/72\\
2/3 & 1/9 & 4/9 & 1/9 & 0\\
1& 1/8 & 3/8 & 3/8 & 1/8\\
\hline
& 1/8 & 3/8 & 3/8 & 1/8 
\end{array}.
\]
In particular we have, using \eqref{eq:defR}, 
$$R(\lambda)={\frac {-{\lambda}^{3}-11\,{\lambda}^{2}-54\,\lambda-108}{{\lambda}^{3
}-11\,{\lambda}^{2}+54\,\lambda-108}}.
$$
Since the points of this method are symmetric with respect to $1/2$ and $\tau(\pi)$
has no root in $\C^-$, we infer (using Remark \ref{rem:CS-Achapeaustable}) that this method
is $\hat{A}$-stable.
\end{example}

\subsection{Examples of Runge--Kutta collocation methods with $s=5$ stages}

\subsubsection{An example of $A$-stable and $AS$-stable method which is not $ASI$-stable}

It may happen that a Runge--Kutta collocation method is $A$-stable and $AS$-stable,
yet not $ASI$-stable, when, for example,
$\tau(\pi(X))$ and $\tau(\pi(X+1))$ share the same roots in $\C^-$.
In this case, $R(\lambda)$ may be smooth and with modulus bounded by $1$ in $\C^-$,
$\lambda\mapsto \lambda b^{t}(I-\lambda A)^{-1}$ be bounded in $\C^-$,
while $\lambda\mapsto (I-\lambda A)^{-1}$ is not bounded in $\C^-$.
This is for example the case for the following $5$-stages method.

\begin{example}[An example of $A$-stable and $AS$-stable method which is not $ASI$-stable]
  \label{ex:5etagesAASpasASI}
  Take $c_1=1/4$, $c_3=1/2$, $c_5=3/4$,
  \begin{equation*}
    c_2=\frac12-\dfrac{\sqrt{7}}{14},
\end{equation*}
and $c_4=1-c_2$. The value of $c_2$ is chosen as a solution to
\begin{equation*}
  14\,c_{{2}}^{2}-14\,c_{
{2}}+3 =0,
\end{equation*}
so that the greatest common divisor of $\tau(\pi(X))$ and $\tau(\pi(X+1))$ is not $1$.
Indeed, the five roots of $\tau(\pi(X))$ are simple and consist in
$i\alpha$, $-i \alpha$, a positive real number 
and two other conjugated complex numbers of positive real part, where
\begin{equation*}
  \alpha=\dfrac{3\sqrt{7}}{56}.
\end{equation*}
Moreover, $i\alpha$ and $-i\alpha$ are also roots of $\tau(\pi(X+1))$.
In particular, using \eqref{eq:Rtaudepi},
$\lambda\mapsto R(\lambda)$ is an analytic function in an open set containing $\C^-$.
Thanks to Section \ref{subsec:symmetry},
the fact that the points are symmetric with respect to $1/2$ implies that
$\lim_{|\lambda|\to+\infty}|R(\lambda)|= 1$ and $|R(\lambda)|=1$ for $\lambda\in \icomplexe\R$.
The maximum principle implies that $|R(\lambda)|\leq 1$
for $\lambda\in\C^-$. Hence the method is $A$-stable.
However, since $\icomplexe\alpha$ and $-\icomplexe\alpha$ are eigenvalues of $A$, we have
\begin{equation*}
  \|(I-\lambda A)^{-1}\| \underset{\lambda \to -\icomplexe\alpha^{-1}}{\longrightarrow} + \infty
  \qquad {\rm and} \qquad
  \|(I-\lambda A)^{-1}\| \underset{\lambda \to \icomplexe\alpha^{-1}}{\longrightarrow} + \infty,
\end{equation*}
so that the method is {\bf not} $ASI$-stable. The boundedness of $\lambda\mapsto \lambda b^{t}(I-\lambda A)^{-1}$ over $\C^-$ follows from the fact that this smooth function is bounded at infinity (degree
at most $-4$ in $\lambda$) and in the neighborhood $\pm i/\alpha$ since $b$ is orthogonal to the eigenvectors of $A$ associated with $\pm i\alpha$. Therefore the method is $AS$-stable.

The full Butcher tableau of this method is
\begin{equation*}
    \renewcommand\arraystretch{1.2}
\begin{array}{c|ccccc}
\frac14 & {\frac{3259}{1440}}&-{\frac{1421}{720}}-
{\frac {21\,\sqrt {7}}{64}}&{\frac{163}{120}}&-{\frac{1421}{720}}+{
\frac {21\,\sqrt {7}}{64}}&{\frac{829}{1440}}\\
\frac12-\frac{\sqrt{7}}{14} & {\medskip}{
\frac { \kappa_{-} ^{2} \left( 281\,\sqrt {7}+1120
 \right) }{15435}}&-{\frac{343}{180}}-{\frac {107\,\sqrt {7}}{315}}&{
\frac { \kappa_{-} ^{2} \left( 106\,\sqrt {7}+455
 \right) }{10290}}&-{\frac { \kappa_{-} ^{2} \left( 
97\,\sqrt {7}+770 \right) }{17640}}&{\frac { \kappa_{-} ^{2} \left( 71\,\sqrt {7}+280 \right) }{15435}}\\
\frac12 & {\medskip}{\frac{203}{90}}&-{\frac{343}{180}}-{\frac {7\,
\sqrt {7}}{24}}&{\frac{22}{15}}&-{\frac{343}{180}}+{\frac {7\,\sqrt {7
                                 }}{24}}&{\frac{53}{90}}\\
\frac12+\frac{\sqrt{7}}{14} & {\medskip}-{\frac { \kappa_{+} ^{2} \left( 281\,\sqrt {7}-1120 \right) }{15435}}&{\frac {
 \kappa_{+} ^{2} \left( 97\,\sqrt {7}-770 \right) }{
17640}}&-{\frac { \kappa_{+} ^{2} \left( 106\,\sqrt {7
}-455 \right) }{10290}}&-{\frac{343}{180}}+{\frac {107\,\sqrt {7}}{315
}}&-{\frac { \kappa_{+} ^{2} \left( 71\,\sqrt {7}-280
    \right) }{15435}}\\
\frac34 & {\medskip}{\frac{363}{160}}&-{\frac{147}{
80}}-{\frac {21\,\sqrt {7}}{64}}&{\frac{63}{40}}&-{\frac{147}{80}}+{
                                                  \frac {21\,\sqrt {7}}{64}}&{\frac{93}{160}}\\
  \hline
  & \frac{128}{45} & \frac{-343}{90} & \frac{44}{15} & \frac{-343}{90} & \frac{128}{45}
\end{array},
\end{equation*}
where $\kappa_{\pm}=7\pm\sqrt{7}$.
\end{example}

\subsubsection{A collocation Runge--Kutta method with $s=5$ stages which is $I$- but not $A$-stable}
We proved Theorem \ref{th:1234stagesIimpliqueA} in Section \ref{subsec:forward1234},
which states that forward Runge--Kutta collocation methods with at most $4$ stages that are
$I$-stable are also $A$-stable.
The example below shows that this is no longer the case for forward Runge--Kutta collocation methods
with at least $5$ stages: There exists a (forward) Runge--Kutta collocation method with $s=5$ stages
which is $I$-stable yet not $A$-stable.

\begin{example}[A collocation Runge--Kutta method with $s=5$ stages which is $I$- but not $A$-stable]
  \label{ex:IstablemaispasAstable}
  We consider the 5-stages Runge--Kutta collocation method defined by its Butcher tableau
  \begin{equation*}
    \renewcommand\arraystretch{1.2}
\begin{array}{c|ccccc}
\frac14 & \frac{4453}{2400} & -\frac{4347}{1600} & \frac{221}{120} & -\frac{1917}{1600} & \frac{1123}{2400}\\
\frac13 & \frac{3824}{2025} & -\frac{133}{50} & \frac{742}{405} & -\frac{179}{150} & \frac{944}{2025}\\
\frac12 & \frac{281}{150} & -\frac{513}{200} & \frac{29}{15} & -\frac{243}{200} & \frac{71}{150}\\
\frac23 & \frac{3808}{2025} & -\frac{194}{75} & \frac{824}{405} & -\frac{28}{25} & \frac{928}{2025}\\
\frac34 & \frac{1503}{800} & -\frac{4131}{1600} & \frac{81}{40} & -\frac{1701}{1600} & \frac{393}{800}\\
  \hline
  & \frac{176}{75} & -\frac{189}{50} & \frac{58}{15} & -\frac{189}{50} & \frac{176}{75} 
\end{array}.
\end{equation*}
This method is $I$-stable but not $A$-stable.
Indeed, since the points $(c_i)_{1\leq i\leq 5}$ are symmetric with respect to $1/2$,
we have, using Remark \ref{rem:CS-symmIstable} that this method is $I$-stable.
Note that, since the matrix $A$ has no eigenvalue on the purely imaginary axis,
this method is moreover $\hat{I}$-stable, thanks to Remark \ref{rem:CS-Ichapeaustable}.
Moreover, the characteristic polynomial of $A$, given by 
$34560X^5 - 17280X^4 + 4164X^3 - 642X^2 + 71X - 6$, has 5 simple complex roots including two complex conjugates with a negative real part (approximately $\delta_\pm \simeq-0.0008959473813 \pm 0.1432367668 i$). In this case $\lambda=1/\delta_+$ and $\lambda=1/\delta_-$ are poles of $R(\lambda)$, with negative real part so that the method is not $A$-stable. 
\end{example}

\section{Conclusion}

  This paper deals with stability of classical Runge--Kutta collocation methods.
  For the purpose of providing stability of linearly implicit methods as developed in
  \cite{DL2020} and used in \cite{DL2023PDE}, when applied to nonlinear evolution PDEs,
  we introduce and analyze several notions of stability.
  These notions are grouped in the definition of $\hat{A}$- and $\hat{I}$-stability
  (see Definition \ref{def:stab}).
  We provide sufficient conditions that can be checked algorithmically to ensure
  that these stability notions are fulfilled by a given Runge--Kutta collocation method.
  In particular, we prove that an $A$-stable Runge--Kutta collocation method such that some
  polynomial $\tau(\pi)$ has no root in $\C^-$ is $\hat{A}$-stable
  (see Lemma \ref{lem:AdonneAchapeau}).
  Since $\tau(\pi)$ can be computed from the $s$ collocation points $c_1,\cdots,c_s$,
  and since the fact that it has no root in $\C^-$ can be checked using the Routh--Hurwitz algorithm,
  we derive an algorithmically checkable sufficient condition to ensure that a given $A$-stable
  Runge--Kutta collocation method is $\hat{A}$-stable.

  We also investigate the relations between $\hat{A}$- and $\hat{I}$-stability
  and show that a forward Runge--Kutta collocation method with $s=5$ stages may be $I$-stable
  even if not $A$-stable (see Example \ref{ex:IstablemaispasAstable}),
  while all forward Runge--Kutta collocation
  methods with at most $4$ stages that are $I$-stable are necessarily $A$-stable
  (Theorem \ref{th:1234stagesIimpliqueA}). We also investigate the influence
  of the symmetry of the collocation points with respect to $1/2$ (see Section \ref{subsec:symmetry}).
  A last consequence of our analysis is that
  Runge--Kutta collocation methods at Gauss points are $\hat{A}$- (hence $\hat{I}$-) stable
  (see Theorem \ref{th:GaussAchapeau}).

\section*{Acknowledgments}
This work was partially supported by the Labex CEMPI (ANR-11-LABX-0007-01).

\bibliographystyle{abbrv}
\bibliography{labib}
\end{document}